\pdfoutput=1
\documentclass[11pt, a4paper, twoside,leqno]{amsart}
\usepackage[centering, totalwidth = 380pt, totalheight = 600pt]{geometry}
\usepackage{amssymb, amsmath, amsthm}
\usepackage{microtype, stmaryrd, url, lmodern, eucal}
\usepackage[shortlabels]{enumitem}
\usepackage[latin1]{inputenc}
\usepackage{color}
\definecolor{darkgreen}{rgb}{0,0.45,0}
\usepackage[colorlinks,citecolor=darkgreen,linkcolor=darkgreen]{hyperref}
\usepackage[british]{babel}

\makeatletter
\def\@cite#1#2{[{#1\if@tempswa ,~#2\fi}]}
\makeatother

\DeclareMathAlphabet{\mathbf}{OT1}{cmr}{b}{n}

\usepackage[arrow, matrix, tips, curve, graph, rotate]{xy}
\SelectTips{cm}{10}

\makeatletter
\def\matrixobject@{%
  \edef \next@{={\DirectionfromtheDirection@ }}%
  \expandafter \toks@ \next@ \plainxy@
  \let\xy@@ix@=\xyq@@toksix@
  \xyFN@ \OBJECT@}
\let\xy@entry@@norm=\entry@@norm
\def\entry@@norm@patched{%
  \let\object@=\matrixobject@
  \xy@entry@@norm }
\AtBeginDocument{\let\entry@@norm\entry@@norm@patched}
\makeatother

\newcommand{\twocong}[2][0.5]{\ar@{}[#2] \save ?(#1)*{\cong}\restore}
\newcommand{\twoeq}[2][0.5]{\ar@{}[#2] \save ?(#1)*{=}\restore}
\newcommand{\rtwocell}[3][0.5]{\ar@{}[#2] \ar@{=>}?(#1)+/l 0.2cm/;?(#1)+/r 0.2cm/^{#3}}
\newcommand{\ltwocell}[3][0.5]{\ar@{}[#2] \ar@{=>}?(#1)+/r 0.2cm/;?(#1)+/l 0.2cm/^{#3}}
\newcommand{\ltwocello}[3][0.5]{\ar@{}[#2] \ar@{=>}?(#1)+/r 0.2cm/;?(#1)+/l 0.2cm/_{#3}}
\newcommand{\dtwocell}[3][0.5]{\ar@{}[#2] \ar@{=>}?(#1)+/u  0.2cm/;?(#1)+/d 0.2cm/^{#3}}
\newcommand{\dltwocell}[3][0.5]{\ar@{}[#2] \ar@{=>}?(#1)+/ur  0.2cm/;?(#1)+/dl 0.2cm/^{#3}}
\newcommand{\drtwocell}[3][0.5]{\ar@{}[#2] \ar@{=>}?(#1)+/ul  0.2cm/;?(#1)+/dr 0.2cm/^{#3}}
\newcommand{\dthreecell}[3][0.5]{\ar@{}[#2] \ar@3{->}?(#1)+/u  0.2cm/;?(#1)+/d 0.2cm/^{#3}}
\newcommand{\utwocell}[3][0.5]{\ar@{}[#2] \ar@{=>}?(#1)+/d 0.2cm/;?(#1)+/u 0.2cm/_{#3}}
\newcommand{\dtwocelltarg}[3][0.5]{\ar@{}#2 \ar@{=>}?(#1)+/u  0.2cm/;?(#1)+/d 0.2cm/^{#3}}
\newcommand{\utwocelltarg}[3][0.5]{\ar@{}#2 \ar@{=>}?(#1)+/d  0.2cm/;?(#1)+/u 0.2cm/_{#3}}

\newcommand{\pullbackcorner}[1][dr]{\save*!/#1-1.2pc/#1:(-1,1)@^{|-}\restore}

\newdir{(}{{}*!<0em,-.14em>-\cir<.14em>{l^r}}
\newdir{ (}{{}*!/-5pt/\dir{(}}
\newdir{ >}{{}*!/-5pt/\dir{>}}

\newcommand{\cat}[1]{\mathrm{\mathcal #1}}
\newcommand{\thg}{{\mathord{\text{--}}}}

\newcommand{\spn}[1]{{\langle{#1}\rangle}}

\newcommand{\defeq}{\mathrel{\mathop:}=}
\newcommand{\cd}[2][]{\vcenter{\hbox{\xymatrix#1{#2}}}}

\renewcommand{\phi}{\varphi}

\newcommand{\C}{{\mathcal C}}

\newcommand{\E}{{\mathcal E}}

\newcommand{\G}{{\mathcal G}}
\renewcommand{\H}{{\mathcal H}}

\newcommand{\J}{{\mathcal J}}

\newcommand{\Z}{{\mathcal Z}}

\newcommand{\xtor}[1]{\cdl[@1]{{} \ar[r]|-{\object@{|}}^{#1} & {}}}
\newcommand{\tor}{\ensuremath{\relbar\joinrel\mapstochar\joinrel\rightarrow}}

\makeatletter

\def\hookleftarrowfill@{\arrowfill@\leftarrow\relbar{\relbar\joinrel\rhook}}
\def\twoheadleftarrowfill@{\arrowfill@\twoheadleftarrow\relbar\relbar}
\def\leftbararrowfill@{\arrowdoublefill@{\leftarrow\mkern-5mu}\relbar\mapstochar\relbar\relbar}
\def\Leftbararrowfill@{\arrowdoublefill@{\Leftarrow\mkern-2mu}\Relbar\Mapstochar\Relbar\Relbar}
\def\leftringarrowfill@{\arrowdoublefill@{\leftarrow\mkern-3mu}\relbar{\mkern-3mu\circ\mkern-2mu}\relbar\relbar}
\def\lefttriarrowfill@{\arrowfill@{\mathrel\triangleleft\mkern0.5mu\joinrel\relbar}\relbar\relbar}
\def\Lefttriarrowfill@{\arrowfill@{\mathrel\triangleleft\mkern1mu\joinrel\Relbar}\Relbar\Relbar}

\def\hookrightarrowfill@{\arrowfill@{\lhook\joinrel\relbar}\relbar\rightarrow}
\def\twoheadrightarrowfill@{\arrowfill@\relbar\relbar\twoheadrightarrow}
\def\rightbararrowfill@{\arrowdoublefill@{\relbar\mkern-0.5mu}\relbar\mapstochar\relbar\rightarrow}
\def\Rightbararrowfill@{\arrowdoublefill@{\Relbar\mkern-2mu}\Relbar\Mapstochar\Relbar\Rightarrow}
\def\rightringarrowfill@{\arrowdoublefill@\relbar\relbar{\mkern-2mu\circ\mkern-3mu}\relbar{\mkern-3mu\rightarrow}}
\def\righttriarrowfill@{\arrowfill@\relbar\relbar{\relbar\joinrel\mkern0.5mu\mathrel\triangleright}}
\def\Righttriarrowfill@{\arrowfill@\Relbar\Relbar{\Relbar\joinrel\mkern1mu\mathrel\triangleright}}

\def\leftrightarrowfill@{\arrowfill@\leftarrow\relbar\rightarrow}
\def\mapstofill@{\arrowfill@{\mapstochar\relbar}\relbar\rightarrow}

\renewcommand*\xleftarrow[2][]{\ext@arrow 20{20}0\leftarrowfill@{#1}{#2}}
\providecommand*\xLeftarrow[2][]{\ext@arrow 60{22}0{\Leftarrowfill@}{#1}{#2}}
\providecommand*\xhookleftarrow[2][]{\ext@arrow 10{20}0\hookleftarrowfill@{#1}{#2}}
\providecommand*\xtwoheadleftarrow[2][]{\ext@arrow 60{20}0\twoheadleftarrowfill@{#1}{#2}}
\providecommand*\xleftbararrow[2][]{\ext@arrow 10{22}0\leftbararrowfill@{#1}{#2}}
\providecommand*\xLeftbararrow[2][]{\ext@arrow 50{24}0\Leftbararrowfill@{#1}{#2}}
\providecommand*\xleftringarrow[2][]{\ext@arrow 10{26}0\leftringarrowfill@{#1}{#2}}
\providecommand*\xlefttriarrow[2][]{\ext@arrow 80{24}0\lefttriarrowfill@{#1}{#2}}
\providecommand*\xLefttriarrow[2][]{\ext@arrow 80{24}0\Lefttriarrowfill@{#1}{#2}}

\renewcommand*\xrightarrow[2][]{\ext@arrow 01{20}0\rightarrowfill@{#1}{#2}}
\providecommand*\xRightarrow[2][]{\ext@arrow 04{22}0{\Rightarrowfill@}{#1}{#2}}
\providecommand*\xhookrightarrow[2][]{\ext@arrow 00{20}0\hookrightarrowfill@{#1}{#2}}
\providecommand*\xtwoheadrightarrow[2][]{\ext@arrow 03{20}0\twoheadrightarrowfill@{#1}{#2}}
\providecommand*\xrightbararrow[2][]{\ext@arrow 01{22}0\rightbararrowfill@{#1}{#2}}
\providecommand*\xRightbararrow[2][]{\ext@arrow 04{24}0\Rightbararrowfill@{#1}{#2}}
\providecommand*\xrightringarrow[2][]{\ext@arrow 01{26}0\rightringarrowfill@{#1}{#2}}
\providecommand*\xrighttriarrow[2][]{\ext@arrow 07{24}0\righttriarrowfill@{#1}{#2}}
\providecommand*\xRighttriarrow[2][]{\ext@arrow 07{24}0\Righttriarrowfill@{#1}{#2}}

\providecommand*\xmapsto[2][]{\ext@arrow 01{20}0\mapstofill@{#1}{#2}}
\providecommand*\xleftrightarrow[2][]{\ext@arrow 10{22}0\leftrightarrowfill@{#1}{#2}}
\providecommand*\xLeftrightarrow[2][]{\ext@arrow 10{27}0{\Leftrightarrowfill@}{#1}{#2}}

\makeatother

\numberwithin{equation}{section}

\theoremstyle{plain}
\newtheorem{Thm}{Theorem}
\newtheorem*{Thm*}{Theorem}
\newtheorem{Prop}[Thm]{Proposition}

\newtheorem{Lemma}[Thm]{Lemma}

\theoremstyle{definition}
\newtheorem{Defn}[Thm]{Definition}

\newtheorem{Ex}[Thm]{Example}

\begin{document}
\leftmargini=2em 
\title{Inner automorphisms of groupoids}
\author{Richard Garner} 
\address{Department of Mathematics, Macquarie University, NSW 2109, Australia} 
\email{richard.garner@mq.edu.au}

\subjclass[2000]{Primary: }
\date{\today}

\thanks{The support of Australian Research Council grants
  DP160101519 and FT160100393 is gratefully acknowledged.}

\begin{abstract}
  Bergman has given the following abstract characterisation of the
  inner automorphisms of a group $G$: they are exactly those
  automorphisms of $G$ which can be extended functorially along any
  homomorphism $G \rightarrow H$ to an automorphism of $H$. This leads
  naturally to a definition of ``inner automorphism'' applicable to
  the objects of any category. Bergman and Hofstra--Parker--Scott have
  computed these inner automorphisms for various structures including
  $k$-algebras, monoids, lattices, unital rings, and
  quandles---showing that, in each case, they are given by an obvious
  notion of conjugation.

  In this note, we compute the inner automorphisms of groupoids,
  showing that they are exactly the automorphisms induced by
  conjugation by a bisection. The twist is that this result is
  \emph{false} in the category of groupoids and homomorphisms; to make
  it true, we must instead work with the less familiar category of
  groupoids and \emph{comorphisms} in the sense of Higgins and
  Mackenzie. Besides our main result, we also discuss generalisations
  to topological and Lie groupoids, to categories and to partial
  automorphisms, and examine the link with the theory of inverse
  semigroups.
\end{abstract}
\maketitle

\section{Background}

In~\cite{Bergman2012An-inner}, Bergman gave an element-free
characterisation of the inner automorphisms of a group $G$: they are
exactly those automorphisms of $G$ which can be extended in a
functorial way along any group homomorphism $f \colon G \rightarrow H$
to an automorphism of $H$. More precisely, Bergman defines an
\emph{extended inner automorphism} $\beta$ of $G$ to be a family of group
automorphisms $(\beta_f \colon H \rightarrow H)$, one for each
homomorphism $f \colon G \rightarrow H$, with the property that the
square
\begin{equation}\label{eq:16}
  \cd{
    {H} \ar[r]^-{\beta_f} \ar[d]_{g} &
    {H} \ar[d]^{g} \\
    {K} \ar[r]^-{\beta_{gf}} &
    {K}
  }
\end{equation}
commutes for all $f \colon G \rightarrow H$ and
$g \colon H \rightarrow K$. It is easy to see that each $a \in G$
gives rise to an extended inner automorphism by taking
\begin{equation}\label{eq:2}
  \beta_f(x) = f(a)xf(a)^{-1} \quad \text{for all} \quad f \colon G \rightarrow H\rlap{ .}
\end{equation}

Less obvious is the fact, proved in~\cite{Bergman2012An-inner}, that
every extended inner automorphism of $G$ is of this form for a
\emph{unique} $a \in G$; in particular, an automorphism of $G$ is
inner just when it is the component $\beta_{1_G}$ of some extended
inner automorphism. In light of this characterisation, it is natural
for Bergman to define:
\begin{Defn}[\cite{Bergman2012An-inner}]
  \label{def:4}
  An \emph{extended inner automorphism} of an object $G$ of a category
  $\C$ is a family of automorphisms
  $(\beta_f \colon H \rightarrow H)$, one for each map
  $f \colon G \rightarrow H$ in $\C$, which render commutative the
  square~\eqref{eq:16} for each $f \colon G \rightarrow H$ and $g
  \colon H \rightarrow K$.
\end{Defn}
Bergman goes on to characterise the extended inner automorphisms in
the category of $k$-algebras over a field $k$, as well as the extended
inner \emph{endomorphisms} (dropping the requirement of invertibility
of the maps $\beta_f$) for groups and $k$-algebras, along with further
variants such as ``inner derivations''.

In~\cite{Hofstra2018Isotropy}, the calculation of extended inner
automorphism groups is pursued further, with these being worked out in
full for the categories of monoids, abelian groups, lattices, unital
rings, racks, and quandles; in each case, the extended inner
automorphisms capture a natural notion of \emph{conjugation}. Besides
these calculations, \cite{Hofstra2018Isotropy} also draws the
connection between Bergman's extended inner automorphisms and the
\emph{isotropy group}~\cite{Funk2012Isotropy} of a topos $\E$; this is
the universal group object in $\E$ which acts naturally on every
$X \in \E$. The link is that, in the case of a presheaf topos
$\E = [\C, \cat{Set}]$, the isotropy group of $\E$ is the functor
$\Z \colon \C \rightarrow \cat{Grp}$ sending each $X \in \C$ to its
group of extended inner automorphisms.

In this note we show that, in the context of \emph{groupoids},
extended inner automorphisms once again capture a natural notion of
conjugation: not now by a single element, but rather by a suitable
family of elements of a kind which is well-known from the study of Lie
and topological groupoids.
\begin{Defn}
  \label{def:3}
  A \emph{bisection} $\alpha$ of a groupoid $\mathbb{G}$ is an
  $\mathbb{G}_0$-indexed family of morphisms
  $(\alpha_u \colon u \rightarrow \tilde\alpha(u))$ for which the
  function
  $\tilde{\alpha} \colon \mathbb{G}_0 \rightarrow \mathbb{G}_0$ is
  a bijection.
\end{Defn}

(Here, and subsequently, we write $\mathbb{G}_0$ and $\mathbb{G}_1$ for
the sets of objects and morphisms of a groupoid.)
Each bisection induces a conjugation automorphism
$c_\alpha \colon \mathbb{G} \rightarrow \mathbb{G}$ with action on
objects $\tilde{\alpha}$ and action on morphisms
\begin{equation}
  \label{eq:3}
  x \colon u \rightarrow v \qquad \mapsto \qquad \alpha_v \circ x
  \circ \alpha_u^{-1} \colon \tilde{\alpha}(u) \rightarrow
  \tilde{\alpha}(v)\rlap{ ;}
\end{equation}
however, the characterisation of these as the extended inner
automorphisms of groupoids is slightly delicate.
It is \emph{not} true that
$c_\alpha \colon \mathbb{G} \rightarrow \mathbb{G}$ is the
$1_\mathbb{G}$-component of an extended inner automorphism of
$\mathbb{G}$ in the category $\cat{Grpd}$ of small groupoids and
homomorphisms. There is an intuitively clear explanation for this
fact: given a homomorphism
$f \colon \mathbb{G} \rightarrow \mathbb{H}$, we should like to define
$\beta_f = c_{f\alpha}$, as in the case of groups, but there is no
obvious way of defining the pushforward $f\alpha$ of the bisection
$\alpha$ along $f$. While this does not rule out the possibility that
there is some \emph{non}-obvious way of defining the pushfoward, we
find that, in fact:
\begin{Prop}
  \label{prop:1}
  There are no non-trivial extended inner automorphisms in
  $\cat{Grpd}$.
\end{Prop}
\begin{proof}
  Let $\beta$ be an extended inner automorphism of the small groupoid
  $\mathbb{G}$. Consider the coproduct $\mathbb{G} + 1$ of
  $\mathbb{G}$ with the terminal groupoid, and
  $\iota \colon \mathbb{G} \rightarrow \mathbb{G} + 1$ the coproduct
  injection. By~\eqref{eq:16}, we have
  $\beta_\iota \circ \iota = \iota \circ \beta_{1_\mathbb{G}}$, and so
  the automorphism
  $\beta_\iota \colon \mathbb{G} + 1 \rightarrow \mathbb{G} + 1$ must
  map the full subcategory $\mathbb{G}$ of $\mathbb{G} + 1$ into
  itself; thus, to be bijective on objects, it must map the remaining
  object $\star$ of $\mathbb{G} + 1$ to itself.

  Now for any homomorphism $f \colon \mathbb{G} \rightarrow
  \mathbb{H}$ and any $x \in \mathbb{H}$, there is a (unique)
  homomorphism $\spn{f,x} \colon \mathbb{G} + 1 \rightarrow \H$ such
  that $\spn{f,x}\circ \iota = f$ and $\spn{f,x}(\star) = x$. The
  first condition implies using~\eqref{eq:16} that $\beta_f \circ \spn{f,x} = \spn{f,x}
  \circ \beta_\iota$, whence
  \begin{equation}\label{eq:4}
    \beta_f(x) = \beta_f(\spn{f,x}(\star)) =
    \spn{f,x}(\beta_\iota(\star)) = \spn{f,x}(\star) = x
  \end{equation}
  so that each $\beta_f$ is the identity on objects.

  Consider now the groupoid $\mathbb{G} + \J$ and coproduct injection
  $\jmath \colon \mathbb{G} \rightarrow \mathbb{G} + \J$, where here
  $\J$ is the groupoid with two objects and a single isomorphism
  between them. Writing $\alpha \colon \upsilon \rightarrow \nu$ for
  the image of this isomorphism in $\mathbb{G} + \J$, we conclude from
  the fact that
  $\beta_\jmath \colon \mathbb{G} + \J \rightarrow \mathbb{G} + \J$ is
  the identity on objects that $\beta_\jmath(\alpha) = \alpha$.

  Now for any homomorphism
  $f \colon \mathbb{G} \rightarrow \mathbb{H}$ and any arrow
  $a \colon u \rightarrow v \in \mathbb{H}$, there is a (unique)
  homomorphism $\spn{f,a} \colon \mathbb{G} + \J \rightarrow \H$ with
   $\spn{f,a}\circ \iota = f$ and $\spn{f,x}(\alpha) = a$.
  Repeating the calculation~\eqref{eq:4}, \emph{mutatis mutandis}, we
  conclude that $\beta_f(a) = a$ so that each $\beta_f$ is also the
  identity on morphisms.
\end{proof}

Despite this negative result, we \emph{can} exhibit the conjugation
automorphisms $c_\alpha \colon \mathbb{G} \rightarrow \mathbb{G}$ as
the values of extended inner automorphisms: to do so, we need to alter
the kind of morphism that we consider between groupoids. Rather than
homomorphisms of groupoids, we must consider the \emph{comorphisms} of
Higgins and Mackenzie~\cite{Higgins1993Duality}. A comorphism
$f \colon \mathbb{G} \rightsquigarrow \mathbb{H}$ between groupoids
comprises a function $f \colon \mathbb{H}_0 \rightarrow \mathbb{G}_0$
together with the assignation to each $a \colon fu \rightarrow v$ in
$\mathbb{G}$ of a map $f(a)_u \colon u \rightarrow \tilde f(a,u)$ in
$\mathbb{H}$, subject to suitable axioms. As explained
in~\cite{Aguiar1997Internal}, bisections \emph{do} transport along
comorphisms; this rectifies the problem we observed earlier, allowing
us to prove that:

\begin{Thm*}
  The extended inner automorphisms of an object $\mathbb{G}$ of the
  category of small groupoids and comorphisms are in bijection with
  the bisections of $\mathbb{G}$. The extended inner automorphism
  corresponding to a bisection $\alpha$ is given by the family of
  congjuation automorphisms
  $(c_{f\alpha} \mid f \colon \mathbb{G} \rightsquigarrow
  \mathbb{H})$.
\end{Thm*}

This is our main result, and will be proven in
Section~\ref{sec:inner-autom-group} below. Preceding this is
Section~\ref{sec:comorph-bisect}, which sets up the necessary
background on~\cite{Higgins1993Duality}'s notion of comorphism, the
relation to the usual groupoid homomorphisms, and the link with
bisections. Finally, after proving our main result, we describe in
Section~\ref{sec:gener-comm} various natural generalisations---to
topological and Lie groupoids, and to categories---and discuss an
alternative perspective involving inverse semigroups.

\section{Comorphisms and bisections}
\label{sec:comorph-bisect}

The definition of comorphism we give here is not the original one
of~\cite{Higgins1993Duality}, but a reformulation due
to~\cite{Aguiar1997Internal}.
\begin{Defn}
  \label{def:1}
  A \emph{comorphism} (sometimes also called \emph{cofunctor})
  $f \colon \mathbb{G} \rightsquigarrow \mathbb{H}$ between
  small groupoids comprises a function
  $f \colon \mathbb{H}_0 \rightarrow \mathbb{G}_0$
  together with functions
  \begin{align*}
    \textstyle\sum_{v \in \mathbb{G}} \mathbb{G}(fu, v) &
    \quad \rightarrow \quad
    \textstyle \sum_{v' \in \mathbb{H}} \mathbb{H}(u, v') \\
    fu \xrightarrow{a} v & \quad \mapsto \quad u \xrightarrow{f(a)_u}
    \tilde f(a,u)
  \end{align*}
  for each $u \in \mathbb{H}_0$, subject to the following axioms:
  \begin{enumerate}[(i)]
  \item $f\bigl(\tilde{f}(a,u)\bigr) = v$ for all $a \colon fu \rightarrow v$ in $\mathbb{G}$;
  \item $f(1_{fu})_u = 1_u$ for all $u \in \mathbb{H}_0$;
  \item $f(b)_{\tilde f(a,u)} \circ f(a)_u = f(ba)_u$ for all $a \colon fu
    \rightarrow v$ and $b \colon v \rightarrow w$ in $\mathbb{G}$.
  \end{enumerate}
  Comorphisms may be composed in the evident manner, and in this way
  we obtain a category
  $\cat{Grpd}_\mathrm{co}$ of small groupoids and comorphisms.
\end{Defn}

There are two ways in which an ordinary homomorphism of
groupoids can give rise to a comorphism. On the one hand, any
bijective-on-objects homomorphism
$f \colon \mathbb{G} \rightarrow \mathbb{H}$ induces a comorphism
$f_\ast \colon \mathbb{G} \rightsquigarrow \mathbb{H}$ which on
objects acts as the inverse to
$f \colon \mathbb{G}_0 \rightarrow \mathbb{H}_0$, and on maps is
given by the assignation
\begin{equation*}
  f_\ast(u) \xrightarrow a v \quad \mapsto
  \quad 
  u \xrightarrow{f(a)} f(v)\rlap{ .}
\end{equation*}

On the other hand, we can obtain a comorphism from any \emph{discrete
  opfibration}. Recall that a homomorphism of groupoids
$f \colon \mathbb{G} \rightarrow \mathbb{H}$ is a discrete opfibration
if, for each $u \in \mathbb{H}_0$ and map
$a \colon fu \rightarrow v$ in $\mathbb{G}$, there is a unique map
$f(a)_u \colon u \rightarrow \tilde f(a,u)$ in $\mathbb{H}$ whose domain is
$u$ and whose image under $f$ is $a$. In this situation, the action on
objects and unique liftings of arrows for
$f \colon \mathbb{G} \rightarrow \mathbb{H}$ provide the data of a
comorphism $f^\ast \colon \mathbb{H} \rightsquigarrow \mathbb{G}$. 

In fact, as explained in~\cite[Section~4.4]{Aguiar1997Internal}, all
comorphisms are generated from those in the image of $(\thg)_\ast$ and
$(\thg)^\ast$:
\begin{Prop}
  \label{prop:5}
  Any comorphism $f \colon \mathbb{G} \rightsquigarrow \mathbb{H}$ can
  be decomposed as
  \begin{equation}\label{eq:6}
    \cd{
      f = \mathbb{G} \ar@{~>}[r]^-{(f_1)^\ast} & \mathbb{K}
      \ar@{~>}[r]^-{(f_2)_\ast} & \mathbb{H}\rlap{ .}
    }
  \end{equation}
\end{Prop}
\begin{proof}
  Let $\mathbb{K}$ be the groupoid whose objects are those of
  $\mathbb{H}$ and whose maps $u \tor v$ are maps
  $a \colon fu \rightarrow fv$ of $\mathbb{G}$ such that $\tilde f(a,u) = v$.
  Composition is inherited from $\mathbb{G}$, and is well-defined by
  axioms (ii) and (iii) for a comorphism. There is an
  identity-on-objects homomorphism
  $f_2 \colon \mathbb{K} \rightarrow \mathbb{H}$ given on maps by
  $f_2(a \colon u \tor v) = f(a)_u \colon u \rightarrow v$; and so we
  can form $(f_2)_\ast \colon \mathbb{K} \rightsquigarrow \mathbb{H}$.
  There is also a homomorphism
  $f_1 \colon \mathbb{K} \rightarrow \mathbb{G}$ with the same action
  as $f$ on objects, and action on morphisms
  $f_1(a \colon u \tor v) = a \colon fu \rightarrow fv$. Note that,
  for any map $a \colon fu \rightarrow v$ in $\mathbb{G}$, the unique
  map of $\mathbb{K}$ with domain $u$ whose $f_1$-image is $a$ is
  $a \colon u \tor \tilde f(a,u)$. So $f_1$ is a discrete opfibration, and we
  can form $(f_1)^\ast \colon \mathbb{G} \rightsquigarrow \mathbb{K}$.
  It is now direct from the definitions that
  $f = (f_2)_\ast (f_1)^\ast$ as in~\eqref{eq:6}.
\end{proof}
In fact, we can equally define comorphisms
$\mathbb{G} \rightsquigarrow \mathbb{H}$ as (equivalence classes of)
spans $\mathbb{G} \leftarrow \mathbb{K} \rightarrow \mathbb{H}$ with
left leg a discrete opfibration and right leg bijective-on-objects;
this is essentially the original definition
of~\cite{Higgins1993Duality}. The following ``Beck--Chevalley lemma''
shows that composition of comorphisms corresponds to the composition
of the representing spans by pullback.
\begin{Lemma}
  \label{lem:2}
  Given a commuting square of homomorphisms
  \begin{equation*}
    \cd{
      {\mathbb{G}} \ar[r]^-{f} \ar[d]_{h} &
      {\mathbb{H}} \ar[d]^{k} \\
      {\mathbb{K}} \ar[r]^-{g} &
      {\mathbb{L}}
    }
  \end{equation*}
  where both $f$ and $g$ are bijective on objects and both $h$ and $k$
  are discrete opfibrations\footnote{Such a square is necessarily a pullback.}, we have
  $f_\ast h^\ast = k^\ast g_\ast \colon \mathbb{K} \rightsquigarrow
  \mathbb{H}$.
\end{Lemma}
\begin{proof}
  On objects, the two composites act by $x \mapsto g^{-1}k(x)$ and
  $x \mapsto hf^{-1}(x)$; these coincide since $gh = kf$ on objects.
  For a map $a \colon g^{-1}k(x) = hf^{-1}(x) \rightarrow y$ in
  $\mathbb{K}$, its image $(k^\ast g_\ast)(a)_x$ is the unique map of
  $\mathbb{H}$ with domain $x$ and $k$-image $g(a)$. On the other
  hand, $(f_\ast h^\ast)(a)_x = f(a')$, where
  $a'$ is the unique map of $\mathbb{G}$ with domain
  $f^{-1}(x)$ and $h$-image $a$. It follows that $f(a')$ has domain $x$ and
  $k$-image $kf(a'') = gh(a'') = g(a)$, and so we
  have $(f_\ast
  h^\ast)_x(a) = (k^\ast g_\ast)_x(a)$ as desired.
\end{proof}

Since we will be interested in automorphisms in the category
$\cat{Grpd}_\mathrm{co}$, the following lemma will be useful; its
straightforward proof is left to the reader.
\begin{Lemma}
  \label{lem:4}
  If $f \colon \mathbb{G} \rightsquigarrow \mathbb{H}$ is an
  invertible map in $\cat{Grpd}_\mathrm{co}$, then 
  $f = g_\ast$ for a unique invertible functor
  $g \colon \mathbb{G} \rightarrow \mathbb{H}$; moreover, we have
  $f^{-1} = g^\ast$.
\end{Lemma}

Finally in this section, we discuss the relationship between
comorphisms and bisections. This is most clearly expressed in the terms
of the fully faithful functor
\begin{equation*}
  \Sigma \colon \cat{Grp} \rightarrow \cat{Grpd}_\mathrm{co}
\end{equation*}
from the category of groups which on objects takes $G$ to the
corresponding one-object groupoid $\Sigma G$. We will show that the
taking of bisections provides a right adjoint to this functor.

First observe that the set $\mathrm{Bis}(\mathbb{G})$ of bisections of
a groupoid $\mathbb{G}$ is indeed a group under the operation on
bisections $\beta, \alpha \mapsto \beta \cdot \alpha$ given by
\begin{equation*}
  (\beta \cdot \alpha)_u = u \xrightarrow{\alpha_u} \tilde \alpha(u) \xrightarrow{\beta_{\tilde
  \alpha(u)}} \tilde \beta(\tilde \alpha(u))\rlap{ .}
\end{equation*}
The identity element is the bisection $\mathbf 1$ with
$(\mathbf 1)_u = 1_u$. The inverse of the bisection $\alpha$ is the
bisection $\alpha^{-1}$ determined by
$(\alpha^{-1})_{\tilde \alpha(u)} = (\alpha_u)^{-1}$. 

\begin{Prop}
  \label{prop:4}
  The full embedding 
  $\Sigma \colon \cat{Grp} \rightarrow \cat{Grpd}_\mathrm{co}$ has a
  right adjoint whose value at a groupoid $\mathbb{G}$ is given by
  the group of bisections $\mathrm{Bis}(\mathbb{G})$.
\end{Prop}
\begin{proof}
  Let $H$ be a group and $\mathbb{G}$ a groupoid, and consider what it
  is to give a comorphism $f \colon \Sigma H \rightarrow \mathbb{G}$.
  On objects, $f$ must send each object $u \in \mathbb{G}$ to the
  unique object $\ast$ of $\Sigma H$. On morphisms, we must give an
  assignation
  \begin{equation*}
    \bigl(f(u) \xrightarrow{a} \ast\bigr) \quad \mapsto \quad \bigl(u
    \xrightarrow{f(a)_u} \tilde f(a,u)\bigr)\rlap{ ,}
  \end{equation*}
  where $a \in H$ and $u \in \mathbb{G}_0$, subject to the axioms
  (i)--(iii) of Definition~\ref{def:1}. Axiom~(i) is trivial as
  $\Sigma H$ has only one object. Axioms (ii) and (iii) state that
  \begin{equation}\label{eq:17}
    f(1_H)_u = 1_u  \qquad \text{and} \qquad 
    f(ba)_u = f(b)_{\tilde{f}(a,u)} \circ f(a)_u
  \end{equation}
  for all $a,b \in H$ and all $u \in \mathbb{H}_0$. Taking
  codomains, we have
  \begin{equation*}
    \tilde f(1_H,u) = u \qquad \text{and} \qquad \tilde f(ba,u) =
    \tilde f(b, \tilde f(a,u))
  \end{equation*}
  so that the assignation $(a, u) \mapsto \tilde f(a, u)$ is an
  $H$-action on $\mathbb{G}_0$. In particular, for each $a \in H$
  the function $u \mapsto \tilde f(a,u)$ is invertible, so that the
  collection of maps
  $\bigl(f(a)_u \colon u \rightarrow \tilde f(a,u)\bigr)$ constitutes
  a bisection $\bar f(a) \in \mathrm{Bis}(\mathbb{G})$. In these
  terms, the conditions~\eqref{eq:17} state precisely that the mapping
  $\bar f \colon H \rightarrow \mathrm{Bis}(\mathbb{G})$ so obtained
  is a group homomorphism. In this way, we have produced bijections
  \begin{equation}\label{eq:18}
    \bigl(\,\overline{\phantom{h}}\,\bigr) \colon \cat{Grpd}_\mathrm{co}(\Sigma H, \mathbb{G}) \rightarrow
    \cat{Gpd}\bigl(H, \mathrm{Bis}(\mathbb{G})\bigr)\rlap{ .}
  \end{equation}
  which are easily seen to be natural in $H$. This shows, as claimed,
  that $\Sigma$ has a right adjoint whose value at the groupoid
  $\mathbb{G}$ is given by $\mathrm{Bis}(\mathbb{G})$.
\end{proof}

A consequence of the adjointness exhibited above is that the
assignation $\mathbb{G} \mapsto \mathrm{Bis}(\mathbb{G})$ extends
uniquely to a functor
$\mathrm{Bis} \colon \cat{Grpd}_\mathrm{co} \rightarrow \cat{Grp}$
making the bijections~\eqref{eq:18} natural in $\mathbb{G}$ as well as
$H$. In other words, there is a canonical way of transporting
bisections along comorphisms. To read off an explicit formula for this
transport, note that, since bisections of $\mathbb{G}$ correspond
bijectively to group homomorphisms
$\mathbb{Z} \rightarrow \mathrm{Bis}(\mathbb{G})$, they also
correspond bijectively to comorphisms
$\Sigma \mathbb{Z} \rightsquigarrow \mathbb{G}$. In these terms, the
transport of a such a bisection along a comorphism
$\mathbb{G} \rightsquigarrow \mathbb{H}$ is given simply by
postcomposition. Spelling this out, we obtain:

\begin{Defn}
  \label{def:2}
  Given a a comorphism
  $f \colon \mathbb{G} \rightsquigarrow \mathbb{H}$ and a bisection
  $\alpha$ of $\mathbb{G}$, the \emph{pushforward bisection} $f\alpha$
  of $\mathbb{H}$ is defined by
  $(f\alpha)_u = f(\alpha_{fu})_u \colon u \rightarrow
  \tilde f(\alpha_{fu},u)$.
\end{Defn}

In particular, if $f \colon \mathbb{G} \rightarrow \mathbb{H}$ is a
bijective-on-objects functor, then pushing forward the bisection
$\alpha$ of $\mathbb{G}$ along $f_\ast$ yields the bisection
$f_\ast \alpha$ of $\mathbb{H}$ whose components are determined by
\begin{equation}\label{eq:7}
  (f_\ast \alpha)_{fu} = f(\alpha_{u})\rlap{ .}
\end{equation}
On the other hand, if $f \colon \mathbb{H} \rightarrow \mathbb{G}$ is
a discrete opfibration, then we can push forward $\alpha$ along
$f^\ast$ to obtain the bisection $f^\ast \alpha$ of $\mathbb{H}$
uniquely determined by
\begin{equation}\label{eq:15}
  f\big((f^\ast \alpha)_u\big) = \alpha_{fu}\rlap{ .}
\end{equation}

\section{Inner automorphisms of groupoids}
\label{sec:inner-autom-group}

We now have all the necessary background to prove our main result. We
begin with the easier direction. 
%
%
\begin{Prop}
  \label{prop:2}
  Each bisection $\alpha$ of the groupoid $\mathbb{G}$ gives an extended
  inner automorphism of $\mathbb{G}$ in
  $\cat{Grpd}_\mathrm{co}$ whose component at $f \colon \mathbb{G}
  \rightsquigarrow \mathbb{H}$ is the conjugation isomorphism
  $( c_{f\alpha} )_\ast \colon \mathbb{H} \rightsquigarrow
  \mathbb{H}$. 
\end{Prop}

\begin{proof}
  We must check that, for each bisection $\alpha$ of a groupoid $\mathbb{G}$, and each
  $f \colon \mathbb{G} \rightsquigarrow
  \mathbb{H}$ and $g \colon \mathbb{H} \rightsquigarrow \mathbb{K}$, the
  square of comorphisms left below commutes:
  \begin{equation}\label{eq:9}
    \cd[@C+2em]{
      {\mathbb{H}} \ar@{~>}[r]^-{(c_{f\alpha} )_\ast} \ar@{~>}[d]_{g} &
      {\mathbb{H}} \ar@{~>}[d]^{g} \\
      {\mathbb{K}} \ar@{~>}[r]^-{( c_{gf\alpha} )_\ast} &
      {\mathbb{K}}} \qquad  \qquad 
    \cd[@C+2em]{
      {\mathbb{G}} \ar@{~>}[r]^-{( c_{\alpha} )_\ast} \ar@{~>}[d]_{g} &
      {\mathbb{G}} \ar@{~>}[d]^{g} \\
      {\mathbb{K}} \ar@{~>}[r]^-{( c_{g\alpha} )_\ast} &
      {\mathbb{K}}\rlap{ .}
    }
  \end{equation}
  Since $f\alpha$ is a bisection of $\mathbb{H}$, we can without loss
  of generality assume that $\mathbb{H} = \mathbb{G}$ and $f = 1$, and
  so reduce to checking commutativity as right above. By
  Proposition~\ref{prop:5} we can in turn reduce to the cases where
  $g = f_\ast$ or where $g = f^\ast$. If $g = f_\ast$ for a
  bijective-on-objects $f$, then we need only check commutativity to
  the left in:
  \begin{equation*}
        \cd[@C+2em]{
      {\mathbb{G}} \ar[r]^-{c_{\alpha}} \ar[d]_{f} &
      {\mathbb{G}} \ar[d]^{f} \\
      {\mathbb{K}} \ar[r]^-{c_{f_\ast(\alpha)}} &
      {\mathbb{K}}
    }\qquad\qquad 
       \cd[@C+2em]{
      {\mathbb{K}} \ar[r]^-{c_{f^\ast(\alpha)}} \ar[d]_{f} &
      {\mathbb{K}} \ar[d]^{f} \\
      {\mathbb{G}} \ar[r]^-{c_{\alpha}} &
      {\mathbb{G}}\rlap{ ;}
    }
  \end{equation*}
  and this holds at a map $a \colon u \rightarrow v$ of
  $\mathbb{G}$ since, by functoriality of $f$ and~\eqref{eq:7},
  \begin{equation*}
      f(\alpha_v \circ a \circ \alpha_u^{-1}) = f(\alpha_v) \circ fa
      \circ f(\alpha_u)^{-1} = (f_\ast\alpha)_v \circ fa \circ (f_\ast
      \alpha)_u\rlap{ .}
  \end{equation*}
  On the other hand, if $g = f^\ast$ for a discrete opfibration $f$,
  then on replacing the horizontal maps $(c_\alpha)_\ast$ and
  $(c_{f^\ast\alpha})_\ast$ in~\eqref{eq:9} by their inverses
  $(c_\alpha)^\ast$ and $(c_{f^\ast\alpha})^\ast$, we may reduce to
  checking commutativity of the square right above. This equality is verified
  at $a \colon u \rightarrow v$ in $\mathbb{K}$ since, by
  functoriality of $f$ and~\eqref{eq:15},
  \begin{equation*}
    f\big( (f^\ast \alpha)_v \circ a \circ (f^\ast \alpha)_u^{-1}\big)
    = f\big( (f^\ast \alpha)_v\big) \circ fa \circ f\big((f^\ast \alpha)_u^{-1}\big)
    = \alpha_{fv} \circ fa \circ \alpha_{fu}^{-1}\text{ .} \qedhere
  \end{equation*}
\end{proof}
It remains to show that:
\begin{Prop}
  \label{prop:3}
  Each extended inner automorphism of $\mathbb{G}$ in
  $\cat{Grpd}_\mathrm{co}$ is induced in the manner of
  Proposition~\ref{prop:2} from a \emph{unique} bisection $\alpha$ of
  $\mathbb{G}$.
\end{Prop}
\begin{proof}
  Suppose we are given an extended inner automorphism $\beta$ of
  $\mathbb{G}$ with components
  $(\beta_f)_\ast \colon \mathbb{H} \rightsquigarrow \mathbb{H}$. To
  prove the result, we must exhibit a unique bisection $\alpha$ of
  $\mathbb{G}$ such that $\beta_f = c_{f\alpha}$ for each $f$.

  We first construct $\alpha$. For each $u \in \mathbb{G}$, consider
  the coslice groupoid $u / \mathbb{G}$, whose objects are arrows
  $a \colon u \rightarrow v$ of $\mathbb{G}$ with domain $u$, and
  whose morphisms are commuting triangles under $u$. The obvious
  codomain projection
  $\pi_u \colon u / \mathbb{G} \rightarrow \mathbb{G}$ is a discrete
  opfibration, and so among the data of $\beta$ is an automorphism
  $\beta_{\pi_u^\ast} \colon u / \mathbb{G} \rightarrow u
  /\mathbb{G}$. Let $\alpha_u \colon u \rightarrow \tilde \alpha(u)$
  be the image of $1_u \in u / \mathbb{G}$ under $\beta_{\pi_u^\ast}$.

  Now as $\beta$ is an extended inner automorphism, the square of
  comorphisms  left below commutes. Replacing
  $(\beta_{1_\mathbb{G}})_\ast$ and $(\beta_{\pi_u^\ast})_\ast$ by
  their inverses $(\beta_{1_\mathbb{G}})^\ast$ and
  $(\beta_{\pi_u^\ast})^\ast$, this is to say that the square of
  homomorphisms to the right below commutes. (Henceforth we will make
  such reductions to homomorphisms without comment.)
  \begin{equation}\label{eq:12}
    \cd{
      {u / \mathbb{G}} \ar@{~>}[r]^-{(\beta_{\pi_u^\ast})_\ast} \ar@{<~}[d]_{(\pi_u)^\ast} &
      {u / \mathbb{G}} \ar@{<~}[d]^{(\pi_u)^\ast} & &
      {u / \mathbb{G}} \ar[r]^-{\beta_{\pi_u^\ast}} \ar[d]_{\pi_u} &
      {u / \mathbb{G}} \ar[d]^{\pi_u} \\
      {\mathbb{G}} \ar@{~>}[r]^-{(\beta_{1_\mathbb{G}})_\ast} &
      {\mathbb{G}} & &
      {\mathbb{G}} \ar[r]^-{\beta_{1_\mathbb{G}}} &
      {\mathbb{G}}
    }
  \end{equation}
  Tracing $1_u$ around this square yields
  $\beta_{1_\mathbb{G}}(u) = \tilde \alpha(u)$. Thus, since
  $\beta_{1_\mathbb{G}}$ is invertible, so is the function
  $u \mapsto \tilde \alpha(u)$; whence $(\alpha_u)_{u \in \mathbb{G}}$
  is a bisection of $\mathbb{G}$.

  We now show that
  $\beta_{1_\mathbb{G}} = c_{\alpha} \colon \mathbb{G} \rightarrow
  \mathbb{G}$. Consider a map $a \colon u \rightarrow v$ of
  $\mathbb{G}$. This induces by precomposition a functor
  $(\thg) \circ a \colon v / \mathbb{G} \rightarrow u / \mathbb{G}$,
  which fits into a commuting triangle of discrete opfibrations as left below. Since $\beta$ is
  an extended inner automorphism, this implies the commutativity of the square of
  homomorphisms to the right.
  \begin{equation*}
    \cd[@C-0.8em@!C]{
      {v / \mathbb{G}} \ar[rr]^-{(\thg) \circ a} \ar[dr]_-{\pi_v} & &
      {u / \mathbb{G}} \ar[dl]^-{\pi_u} \\ &
      {\mathbb{G}}
    } \qquad  \qquad 
    \cd{
      {v/\mathbb{G}} \ar[r]^-{\beta_{\pi_{v}^\ast}} \ar[d]_{(\thg)
        \circ a} &
      {v/\mathbb{G}} \ar[d]^{(\thg) \circ a} \\
      {u/\mathbb{G}} \ar[r]^-{\beta_{\pi_{u}^\ast}} &
      {u/\mathbb{G}}
    } 
  \end{equation*}

  Tracing $1_v$ around this square, we find that
  $\beta_{\pi_u^\ast}$ sends the object $a \in u / \mathbb{G}$
  to $\alpha_v \circ a \in u / \mathbb{G}$. Thus
  $\beta_{\pi_u^\ast}$ sends the \emph{map}
  $a \colon 1_u \rightarrow a$ of $u / \mathbb{G}$ to a map
  $f \colon \alpha_u \rightarrow \alpha_v \circ a$ of
  $u / \mathbb{G}$. Note that
  $\pi_u(f) \colon \tilde \alpha(u) \rightarrow \tilde \alpha(v)$
  satisfies $\pi_u(f) \circ \alpha_u = \alpha_v \circ a$, and so necessarily
  $\pi_u(f) = \alpha_v \circ a \circ \alpha_u^{-1}$. Thus, tracing  $a \colon 1_u \rightarrow a$ around the right square
  of~\eqref{eq:12}, we see that
  $\beta_{1_\mathbb{G}}(a \colon u \rightarrow v) = \alpha_v \circ a
  \circ \alpha_u^{-1}$, and so $\beta_{1_\mathbb{G}} = c_\alpha$ as
  claimed.

  It remains to show that $\beta_f = c_{f\alpha}$ for \emph{all}
  comorphisms $f \colon \mathbb{G} \rightsquigarrow \mathbb{H}$. For
  this, it suffices to show that the bisection associated to the
  extended inner automorphism $\beta_{(\thg)f}$ of $\mathbb{H}$ is
  $f\alpha$, since then
  $\beta_f = \beta_{(1_\mathbb{H})f} = c_{f\alpha}$ as desired. That is, we must prove:
  \begin{equation}
    \label{eq:14}
    \beta_{\pi_u^\ast f}(1_u) = (f\alpha)_u \qquad \text{for all } f
    \colon \mathbb{G} \rightsquigarrow \mathbb{H} \text{ and } u
    \in \mathbb{H}_0\rlap{ .}
  \end{equation}
  
  \textbf{Step 1}. Suppose first that $f = g^\ast$ for some
  discrete opfibration $g \colon \mathbb{H} \rightarrow \mathbb{G}$.
  We then have a commuting square of discrete opfibrations as to the
  left in:
  \begin{equation*}
    \cd[@!@-0.8em]{
      {u/\mathbb{H}} \ar[r]^-{u/g} \ar[d]_{\pi_u} &
      {gu/\mathbb{G}} \ar[d]^{\pi_{gu}} & & 
      u/\mathbb{H} \ar[r]^-{\beta_{\pi_{u}^\ast g^\ast}} \ar[d]_-{u/g}&
      u/\mathbb{H} \ar[d]^-{u/g}\\
      {\mathbb{H}} \ar[r]^-{g} &
      {\mathbb{G}} & &
      gu/\mathbb{G} \ar[r]^-{\beta_{\pi_{gu}^\ast}} &
      gu/\mathbb{G}\rlap{ ,}
    }
  \end{equation*}
  and so, since $\beta$ is an extended inner automorphism, a commuting
  square of homomorphisms as to the right. Tracing $1_u$ around this
  square yields
  $g(\beta_{\pi_u^\ast g^\ast}(1_u)) = \beta_{\pi^\ast_{gu}}(1_{gu}) =
  \alpha_{gu}$, and so by~\eqref{eq:15} that
  $\beta_{\pi_u^\ast g^\ast}(1_u) = (g^\ast \alpha)_u$ as required
  for~\eqref{eq:14}.

  \textbf{Step 2}. Suppose next that $f = h_\ast$ for some
  bijective-on-objects functor
  $h \colon \mathbb{G} \rightarrow \mathbb{H}$. We form the outer
  square, the pullback and the induced comparison map as to the left
  in:
  \begin{equation*}
     \cd[@-1.8em@!R@R-0.3em]{
       u / \mathbb{G} \ar[dddd]_-{\pi_{u}} 
       \ar[rrrr]^-{u/h} \ar@{.>}[dr]^-{r} &&&&
       h u / \mathbb{H} \ar[dddd]^-{\pi_{hu}} &&&&
       u / \mathbb{G} \ar[rrrr]^-{\beta_{\pi_u\ast}} \ar[dd]_-{r} &&&&
       u / \mathbb{G} \ar[dd]_-{r} &&&&
       u / \mathbb{G} \ar[rrrr]^-{\beta_{\pi_u\ast}} \ar[dddd]_-{u/h} &&&&
       u / \mathbb{G} \ar[dddd]_-{u/h}
       \\ & 
       \mathbb{P} \ar[dddl]_-{p} \ar[rrru]^-{q} \pullbackcorner \\
       &&&& &&&& \mathbb{P} \ar[rrrr]^-{\beta_{p^\ast}} \ar[dd]_-{q} &&&&
       \mathbb{P} \ar[dd]_-{q} \\ \\
       \mathbb{G}  \ar[rrrr]^-{h} &&&& \mathbb{H} &&&&
       hu / \mathbb{G} \ar[rrrr]^-{\beta_{q_\ast p^\ast}} &&&& hu /
       \mathbb{G} &&&&
       hu / \mathbb{G} \ar[rrrr]^-{\beta_{\pi^\ast_{hu} h_\ast}} &&&& hu / \mathbb{G}
     }
  \end{equation*}
  
  In this square, $\pi_u$ and $\pi_{hu}$ are discrete opfibrations,
  and so is $p$, since it is a pullback of $\pi_{hu}$. It follows that
  the comparison map $r$ is also a discrete opfibration. Since
  $pr = \pi_u$ and $\beta$ is an extended inner automorphism, we have
  that the top square centre above commutes. On the other hand, $q$ is
  bijective-on-objects as a pullback of $f$, and so, since $\beta$ is
  an extended inner automorphism, we have that the bottom square
  centre above commutes.

  By applying Lemma~\ref{lem:2} to the pullback square left above, we
  have $q_\ast p^\ast = \pi_{hu}^\ast h_\ast$, and so the composite of
  the two centre squares is equally the square right above. Tracing
  the object $1_u$ around both sides yields
  $\beta_{\pi_{hu}^\ast h_\ast}(1_{hu}) = h(\alpha_u)$, and so
  by~\eqref{eq:7} we conclude that
  $\beta_{\pi_{hu}^\ast h_\ast}(1_{u}) = (h_\ast\alpha)_u$ as required
  for~\eqref{eq:14}.


  \textbf{Step 3}. We now prove for a general
  $f \colon \mathbb{G} \rightsquigarrow \mathbb{H}$ that the bisection
  associated to the inner automorphism $\beta_{(\thg)f}$ of
  $\mathbb{H}$ is $f\alpha$. We first apply Proposition~\ref{prop:5} to
  decompose $f$ as
  $h_\ast g^\ast \colon \mathbb{G} \rightsquigarrow \mathbb{K}
  \rightsquigarrow \mathbb{H}$. By Step 1 applied to $\beta$ and $g$,
  the bisection associated to the inner automorphism
  $\beta_{(\thg)g^\ast}$ of $\mathbb{K}$ is $g^\ast \alpha$. Now by
  Step 2 applied to $\beta_{(\thg)g^\ast}$ and $h$, the bisection
  associated to the inner automorphism $\beta_{(\thg)h_\ast g^\ast} =
  \beta_{(\thg)f}$ of
  $\mathbb{K}$ is $h_\ast g^\ast \alpha = f\alpha$, as required.
\end{proof}

We have thus proved the theorem stated in the introduction. In fact,
we can do slightly better. The extended inner automorphisms of any
object in any category form a group under the operation of
composition. We noted above that the bisections of a groupoid also
form a group. It is easily seen that these two group structures are
related by the equation
$c_\beta \circ c_\alpha = c_{\beta \cdot \alpha}$, and so we have:
\begin{Thm}
  \label{thm:1}
  The group of extended inner automorphisms of
  $\mathbb{G} \in \cat{Grpd}_\mathrm{co}$ is isomorphic to the group
  $\mathrm{Bis}(\mathbb{G})$. The extended inner automorphism
  corresponding to $\alpha \in \mathrm{Bis}(\mathbb{G})$ is given by the family of
  automorphisms
  $\bigl((c_{f\alpha})_\ast \colon \mathbb{H} \rightsquigarrow
  \mathbb{H}\bigr)$ as $f$ ranges over comorphisms $\mathbb{G} \rightsquigarrow \mathbb{H}$.
\end{Thm}

\section{Generalisations and further perspectives}
\label{sec:gener-comm}

\subsection{Topological and Lie groupoids}
\label{sec:topol-lie-group}

As mentioned in the introduction, bisections show up frequently in the
study of Lie and topological groupoids. It is therefore natural to ask
if our results generalise to those settings. The answer is yes. As
the adaptations in the two cases are so similar, we concentrate on the
topological~one.

First we must adapt the basic notions. For a topological groupoid
$\mathbb{G}$, we restrict attention to \emph{continuous}
bisections~$\alpha$: those for which the assignation
$u \mapsto \alpha_u$ is continuous as a map
$\mathbb{G}_0 \rightarrow \mathbb{G}_1$. This implies, easily, that
the associated conjugation homomorphism
$c_\alpha \colon \mathbb{G} \rightarrow \mathbb{G}$ is a continuous
map of topological groupoids. We should like to identify these
$c_\alpha$'s as the extended inner automorphisms of $\mathbb{G}$ in a
suitable category.

The morphisms of this category will be comorphisms
$f \colon \mathbb{G} \rightsquigarrow \mathbb{H}$ between topological
groupoids which are \emph{continuous}, in the sense of rendering
continuous the following maps of spaces:
\begin{align*}
  \mathbb{H}_0 & \rightarrow \mathbb{G}_0 &
  \mathbb{G}_1 \times_{\mathbb{G}_0} \mathbb{H}_0 & \rightarrow
  \mathbb{H}_1\\
  u & \mapsto fu & (a,u) & \mapsto f(a)_u\rlap{ ;}
\end{align*}
here, the fibre product
$\mathbb{G}_1 \times_{\mathbb{G}_0} \mathbb{H}_0$ is taken along the
source map $s \colon \mathbb{G}_1 \rightarrow \mathbb{G}_0$ and the
action on objects $f \colon \mathbb{H}_0 \rightarrow \mathbb{G}_0$.
Much like before, we obtain continuous comorphisms
$\mathbb{G} \rightsquigarrow \mathbb{H}$ from continuous functors
$\mathbb{G} \rightarrow \mathbb{H}$ which are homeomorphic-on-objects;
and from functors $\mathbb{H} \rightarrow \mathbb{G}$ which are
\emph{continuous discrete opfibrations}, meaning that the operation of
forming the unique lifting $f(a)_u \colon u \rightarrow \tilde f(a,u)$
of a map $a \colon fu \rightarrow v$ is a continuous map
$\mathbb{G}_1 \times_{\mathbb{G}_0} \mathbb{H}_0 \rightarrow
\mathbb{H}_1$. As in Proposition~\ref{prop:5}, every continuous
comorphism arises by composing ones of these two special kinds.

In this situation, we also have an analogue of
Proposition~\ref{prop:4}: the functor
$\Sigma \colon \cat{Grp} \rightarrow \cat{Top\G rpd}_\mathrm{co}$
embedding each (discrete!) group $G$ as a one-object discrete
topological groupoid has a right adjoint, sending a topological
groupoid $\mathbb{G}$ to its discrete group of continuous bisections
$\mathrm{Bis}(\mathbb{G})$\footnote{Note that any attempt to construct
  a right adjoint to the embedding
  $\cat{Top\G rp} \rightarrow \cat{Top\G rpd}_\mathrm{co}$ would fall
  foul of the failure of topological spaces to be cartesian closed.}.
In particular, continuous bisections of a topological groupoid can be
transported along continuous comorphisms, with the same formulae as
before. Using this Proposition~\ref{prop:2} carries over,
\emph{mutatis mutandis}, showing that every continuous bisection of
$\mathbb{G} \in \cat{Top\G rpd}_\mathrm{co}$ induces an extended inner
automorphism.

All that remains is to adapt the proof of Proposition~\ref{prop:3},
showing that every extended inner automorphism $\beta$ of $\mathbb{G}$
arises in this manner. All of the constructions in this proof continue
to work in the topological context, and so we can conclude immediately
that $\beta$ must be of the form $\beta_f = c_{f \alpha}$ for a
unique, but not necessarily continuous, bisection $\alpha$ of
$\mathbb{G}$. To prove continuity, we consider the
\emph{d\'ecalage}~\cite{Illusie1972Complexe} of $\mathbb{G}$. This is
the topological groupoid $\mathrm{Dec}(\mathbb{G})$ whose underlying
topological graph is given by
\begin{equation*}
  \cd{
    \mathbb{G}_1 \,{}_t \!\times_s \mathbb{G}_1 \ar@<3pt>[r]^-{\pi_1}
    \ar@<-3pt>[r]_-{\mu} & \mathbb{G}_1
  }\rlap{ ,}
\end{equation*}
where $\mu$ is the composition map of $\mathbb{G}$. The composition
and units of $\mathrm{Dec}(\mathbb{G})$ itself are determined by
requiring that its underlying discrete groupoid be the disjoint union
of the coslice categories $u / \mathbb{G}$. There is a continuous
discrete opfibration
$\pi \colon \mathrm{Dec}(\mathbb{G}) \rightarrow \mathbb{G}$ which
projects onto the codomain; and for each $u \in \mathbb{G}_0$ this
fits into a commuting triangle as to the left~in
\begin{equation*}
  \cd[@C-1.3em@!C]{
    {u / \mathbb{G}} \ar[rr]^-{\iota} \ar[dr]_-{\pi_u} & &
    {\mathrm{Dec}(\mathbb{G})} \ar[dl]^-{\pi} \\ &
    {\mathbb{G}}
  } \qquad \qquad 
  \cd{
    {u/\mathbb{G}} \ar[r]^-{\beta_{\pi_{u}^\ast}} \ar[d]_{\iota} &
    {u/\mathbb{G}} \ar[d]^{\iota} \\
    {\mathrm{Dec}(\mathbb{G})} \ar[r]^-{\beta_{\pi^\ast}} &
    {\mathrm{Dec}(\mathbb{G})}\rlap{ .}
  } 
\end{equation*}
It follows that each square as right above is commutative in
$\cat{Top\G rpd}$; so in particular,
$\beta_{\pi^\ast}(1_u) = \beta_{\pi_u^\ast}(1_u) = \alpha_u$ for each
$u \in \mathbb{G}_0$. This shows that composing the continuous
identities map
$1_{(\thg)} \colon \mathbb{G}_0 \rightarrow \mathbb{G}_1$ with the
continuous map $\mathbb{G}_1 \rightarrow \mathbb{G}_1$ giving the
action on objects of $\beta_{\pi^\ast}$ yields the assignation
$u \mapsto \alpha_u$---which is thus continuous, as desired. We thus
obtain:
\begin{Thm}
  The group of extended inner automorphisms of
  $\mathbb{G} \in \cat{TopGrpd}_\mathrm{co}$ is isomorphic to the
  group of continuous bisections $\mathrm{Bis}(\mathbb{G})$, under the
  same correspondence as in Theorem~\ref{thm:1}.
\end{Thm}

\subsection{Internal groupoids}
\label{sec:internal-groupoids}

Topological and Lie groupoids are particular examples of
\emph{internal groupoids} in a category $\C$. 
It is therefore natural to ask if our results generalise further
to groupoids internal to any category $\C$. The answer is no.

To see this, consider the category $\cat{Set}^{\mathbb{Z}_2}$ whose
objects are sets $X$ endowed with an involution
$\tau \colon X \rightarrow X$, and whose maps are equivariant
functions (i.e., ones commuting with the involutions). A groupoid
internal to $\cat{Set}^{\mathbb{Z}_2}$ is an ordinary groupoid
$\mathbb{G}$ with a (strict) involution
$\tau \colon \mathbb{G} \rightarrow \mathbb{G}$; internal functors and
comorphisms are just ordinary functors and comorphisms which commute
with the involutions.

Now if $(\mathbb{G}, \tau)$ is an involutive groupoid, then the
functor $\tau$ is easily seen to be equivariant
$(\mathbb{G}, \tau) \rightarrow (\mathbb{G}, \tau)$; it follows that
$(\mathbb{G}, \tau)$ has an extended inner automorphism $\beta$ whose
component at any
$(\mathbb{G}, \tau) \rightsquigarrow (\mathbb{H}, \sigma)$ is
$\sigma_\ast \colon (\mathbb{H}, \sigma) \rightsquigarrow (\mathbb{H},
\sigma)$. However, this $\beta$ need not arise from any bisection
$\alpha$ of $\mathbb{G}$. For example, if $\mathbb{G}$ is the discrete
groupoid on two objects, and $\tau$ is the swap map, then
$(\mathbb{G}, \tau)$ has no non-identity bisections, and yet the $\beta$
defined above is
not the identity.

The reason that things work differently in this case is really that
objects in the indexing category $\mathbb{Z}_2$ can have their own
non-trivial extended inner automorphisms. A more general formulation
of our results would have to take this into account---but, lacking as
we do any compelling reasons for developing such a generalisation, we
have not pursued this further.

\subsection{Categories}
\label{sec:categories}

Another obvious direction of generalisation involves replacing
groupoids everywhere by categories. There is not so much to say here;
everything works without fuss. Comorphisms are defined exactly as
before, and factorise in exactly the same way. For bisections, we must
add the requirement that each map
$\alpha_u \colon u \rightarrow \tilde \alpha(u)$ is invertible, and
can then induce conjugation automorphisms in exactly the same way. Once
again, bisections transport along comorphisms, with this now being
evidenced by an an adjunction
$\Sigma \colon \cat{Grp} \leftrightarrows \cat{Cat}_\mathrm{co} \colon
\mathrm{Bis}$.

Proposition~\ref{prop:2} continues to work; and the only adaptation
required in Proposition~\ref{prop:3} is at the very start. Given an
extended inner automorphism $\beta$ of a category $\mathbb{C}$, with
components $(\beta_f)_\ast$, we have as before the collection of maps
$\alpha_u = \beta_{\pi_u^\ast}(1_u) \colon u \rightarrow \tilde
\alpha(u)$. The argument showing the assignation
$u \mapsto \tilde \alpha(u)$ is invertible still holds; but we must
now also show that each $\alpha_u$ is invertible. For this, we first
show as before that the automorphism
$\beta_{\pi_u^\ast} \colon u / \mathbb{C} \rightarrow u / \mathbb{C}$
is given on objects by
$\beta_{\pi_u^\ast}(a \colon u \rightarrow v) = \alpha_v \circ a$.
Being an automorphism, there is in particular some such $a$ for which
$\alpha_v \circ a = 1_u$. So we have a commuting triangle as to the
left below in $u / \mathbb{C}$. Applying $\beta_{\pi_u^\ast}$ yields a
commuting triangle as to the right.
\begin{equation*}
  \cd[@-0.5em]{
    & {a} \ar@{<-}[dl]_-{a} \ar[dr]^-{\alpha_v} & & &
    & {1_u} \ar@{<-}[dl]_-{f} \ar[dr]^-{g} &
    \\
    {1_u} \ar[rr]_-{1_u} & & {1_u} & &
    {\alpha_u} \ar[rr]_-{1_{\tilde \alpha(u)}} & & {\alpha_u}
  } 
\end{equation*}
Since $g \colon 1_u \rightarrow \alpha_u$ we have
$g = g \circ 1_u = \alpha_u$; since
$f \colon \alpha_u \rightarrow 1_u$ we have $f \circ \alpha_u = 1_u$;
and since the triangle commutes we have $\alpha_u \circ f = 1_u$. So
$f$ is an inverse for $\alpha_u$. The remainder of the argument now
follows exactly as before, and so we have:
\begin{Thm}
  The group of extended inner automorphisms of
  $\mathbb{C} \in \cat{Cat}_\mathrm{co}$ is isomorphic to the group of 
  bisections $\mathrm{Bis}(\mathbb{C})$, under the same correspondence
  as in Theorem~\ref{thm:1}.
\end{Thm}

\subsection{Inverse semigroups}
\label{sec:inverse-semigroups}

Our main results seem to diverge from the pattern for the
computation of extended inner automorphism groups
in~\cite{Bergman2012An-inner, Hofstra2018Isotropy}. In this prior
work, the categories under consideration have as objects, the models
of an equational algebraic theory $\mathbb{T}$, and as morphisms, the
obvious structure-preserving maps. This allows the extended inner
automorphisms of a $\mathbb{T}$-model $X$ to be characterised via
universal algebra: they correspond to those invertible unary
operations of the diagram theory $\mathbb{T}_X$\footnote{i.e., the
  theory obtained by extending $\mathbb{T}$ with new constants for
  each element of $X$, and new equations describing the value of each
  $\mathbb{T}$-operation on those constants.} which
\emph{commute} with each $\mathbb{T}$-operation.

By contrast, our main result concerns the category of groupoids and
comorphisms; and while the objects of this category are algebraic in
nature---they are the models of an \emph{essentially-algebraic} theory
in the sense of~\cite{Freyd1972Aspects}---the morphisms are \emph{not}
the obvious structure-preserving ones (which led only to our negative
Proposition~\ref{prop:1}). This means that our argument for computing
the extended inner automorphisms is necessarily different in nature.

In fact, there is a way of reconciling our results with those
of~\cite{Bergman2012An-inner, Hofstra2018Isotropy}: we adopt a
different perspective on groupoid structure in which the comorphisms
\emph{are} the natural structure-preserving maps. More precisely, we
take as the basic data of a groupoid not its objects and morphisms,
but its \emph{partial} bisections:

\begin{Defn}
  \label{def:7}
  A \emph{partial bisection} $\alpha$ of a groupoid $\mathbb{G}$
  comprises subsets $s(\alpha)$ and $t(\alpha)$ of $\mathbb{G}_0$
  together with an $s(\alpha)$-indexed family of morphisms
  $\bigl(\alpha_u \colon u \rightarrow \tilde \alpha(u)\bigr)$ with the property
  that the assignation $u \mapsto \tilde \alpha(u)$ is a
  bijection $s(\alpha) \rightarrow t(\alpha)$.
\end{Defn}
The set $\mathrm{PBis}(\mathbb{G})$ of partial bisections of a
groupoid $\mathbb{G}$ can be endowed with the structure of a
\emph{pseudogroup}---a special kind of inverse semigroup---and this
structure allows $\mathrm{PBis}(\mathbb{G})$ to represent $\mathbb{G}$
faithfully. This fits into the pattern of a well-known correspondence
between \'etale topological groupoids and pseudogroups, detailed, for
example, in~\cite{Resende2006Lectures,Lawson2013Pseudogroups}. A fact
about this correspondence which does not appear to have been noted
previously is that it equates the natural structure-preserving maps of
pseudogroups with the \emph{comorphisms} between the corresponding
groupoids\footnote{Indeed, in~\cite{Resende2006Lectures}, the
  correspondence is not defined on morphisms; while
  in~\cite{Lawson2013Pseudogroups}, the morphisms considered are on
  one side, the \emph{discrete opfibrations} of groupoids, and on the
  other, a slightly delicate class of morphisms between
  pseudogroups.}. Thus, we may consider our result about groupoids and
comorphisms instead as a result about pseudogroups and their
structure-preserving morphisms, so fitting it in
to the general pattern
established in~\cite{Bergman2012An-inner, Hofstra2018Isotropy}.

To make the preceding claims more precise, we now define the category
$\cat{Ps\G rp}$ of pseudogroups, and sketch a proof that the
assignation $\mathbb{G} \mapsto \mathrm{PBis}(\mathbb{G})$ yields a
full embedding of $\cat{Grpd}_\mathrm{co}$ into $\cat{Ps\G rp}$.

\begin{Defn}
  \label{def:5}
  An \emph{inverse monoid} is a unital semigroup $M$ such that, for
  every $m \in M$, there is a unique $m^\ast \in M$ with
  $mm^\ast m = m$ and $m^\ast mm^\ast = m^\ast$. The \emph{natural
    partial order} $\leqslant$ and the \emph{compatibility relation}
  $\sim$ on $M$ are given by
  \begin{align*}
   m \leqslant n \qquad & \text{iff} \qquad mn^\ast n = n  \\ m \sim
   n \qquad & \text{iff} \qquad mn^\ast \text{ and } n^\ast m \text{ are idempotent.}
 \end{align*}
 A (abstract) \emph{pseudogroup} is an inverse monoid $M$ such that
 any family $S \subseteq M$ of pairwise-compatible elements admits a
 join $\bigvee S$ (with respect to $\leqslant$) which is preserved by
 each function $m \cdot (\thg) \colon M \rightarrow M$ and
 $(\thg) \cdot m \colon M \rightarrow M$. Pseudogroups form a category
 $\cat{Ps\G rp}$ wherein maps are monoid homomorphisms that preserve
 joins of compatible families.
\end{Defn}
\begin{Ex}
  \label{ex:1}
  For any groupoid $\mathbb{G}$, the set of partial bisections
  $\mathrm{PBis}(\mathbb{G})$ is a pseudogroup under the
  binary operation
  $\beta, \alpha \mapsto \beta \cdot \alpha$, where 
  $\beta \cdot \alpha$ has
\begin{equation*}
  s(\beta \cdot \alpha) = s(\alpha) \cap \tilde
  \alpha^{-1}\bigl(s(\beta)\bigr) \qquad
  t(\beta \cdot \alpha) = \tilde \beta\bigl(t(\alpha)\bigr) \cap t(\beta)
\end{equation*}
and components
$(\beta \cdot \alpha)_u = \beta_{\tilde \alpha(u)} \circ \alpha_u
\colon u \rightarrow \tilde \beta(\tilde \alpha(u))$. The unit for
this operation is the identity bisection $\mathbf{1}$, and the partial
inverse $\alpha^\ast$ of $\alpha$ has $s(\alpha^\ast) = t(\alpha)$,
$t(\alpha^\ast) = s(\alpha)$ and components determined by
$(\alpha^\ast)_{\tilde \alpha(u)} = (\alpha_u)^{-1}$.

Two partial bisections $\alpha, \beta$ are compatible if
$\alpha_u = \beta_u$ for all $u \in s(\alpha) \cap s(\beta)$, while
$\alpha \leqslant \beta$ if $\alpha \sim \beta$ and
$s(\alpha) \subseteq s(\beta)$. The join $\alpha$ of a
pairwise-compatible family of partial bisections
$(\alpha^{i} : i \in S)$ has $s(\alpha) = \bigcup_i s(\alpha^i)$,
$t(\alpha) = \bigcup_i t(\alpha^i)$ and components
$\alpha_u = \alpha^i_u$, for any $i \in S$ with
$u \in s(\alpha^i)$.
\end{Ex}

\begin{Prop}
  \label{prop:6}
  The assignation $\mathbb{G} \mapsto \mathrm{PBis}(\mathbb{G})$ is
  the action on objects of a full embedding of categories
  $\cat{Grpd}_\mathrm{co} \rightarrow \cat{Ps\G rp}$.
\end{Prop}
\begin{proof}
  Let $f \colon \mathbb{G} \rightsquigarrow \mathbb{H}$ be a comorphism
  of groupoids, and $\alpha \in \mathrm{PBis}(\mathbb{G})$.
  Generalising Definition~\ref{def:2}, we can define a pushforward
  partial bisection $f\alpha \in \mathrm{PBis}(\mathbb{H})$ by taking
  $s(f\alpha)$ and $t(f\alpha)$ to be the inverse images of
  $s(\alpha)$ and $t(\alpha)$ under the function
  $f \colon \mathbb{H}_0 \rightarrow \mathbb{G}_0$, and with
  components given like before by
  \begin{equation*}
    (f\alpha)_u = f(\alpha_{fu})_u \colon u \rightarrow
    \tilde f(\alpha_{fu},u)\rlap{ .}
  \end{equation*}
  Straightforward checking shows that the assignation
  $\alpha \mapsto f\alpha$ is a pseudogroup morphism
  $\mathrm{PBis}(f) \colon \mathrm{PBis}(\mathbb{G}) \rightarrow
  \mathrm{PBis}(\mathbb{H})$ and that the assignation
  $f \mapsto \mathrm{PBis}(f)$ is functorial; so we have a functor
  $\mathrm{PBis} \colon \cat{Grpd}_\mathrm{co} \rightarrow \cat{Ps\G
    rp}$.

  To see this functor is faithful, note that we can recover the action
  on objects of $f \colon \mathbb{G} \rightsquigarrow \mathbb{H}$ from
  $\varphi \defeq \mathrm{PBis}(f) \colon \mathrm{PBis}(\mathbb{G}) \rightarrow
  \mathrm{PBis}(\mathbb{H})$ by the formula
  \begin{equation}\label{eq:19}
    f(v) = u \qquad \text{iff} \qquad v \in s\bigl(\varphi([1_u])\bigr)\rlap{ ;}
  \end{equation}
  here, if $a \colon u \rightarrow v$ is any map of $\mathbb{G}$ then
  we write $[a]$ for the partial bisection whose sole component is the
  map $a$. In a similar way, we can recover the action of the
  comorphism $f$ on maps by the formula
  \begin{equation}\label{eq:20}
    f(a)_u = b \qquad \text{iff} \qquad \varphi([a])_u = b\rlap{ .}
  \end{equation}

  It remains only to show that $\mathrm{PBis}$ is full. So let
  $\varphi \colon \mathrm{PBis}(\mathbb{G}) \rightarrow
  \mathrm{PBis}(\mathbb{H})$ be any pseudogroup morphism. In
  $\mathrm{PBis}(\mathbb{G})$ we have that
  \begin{equation*}
    \mathbf{1} = \textstyle\bigvee_{u \in \mathbb{G}_0} [1_u] \qquad \text{and}
    \qquad [1_u] \cdot [1_v] = \bot \text{ for $u \neq v$}\rlap{ ;}
  \end{equation*}
  since $\varphi$ is a pseudogroup morphism, it follows that in
  $\mathrm{PBis}(\mathbb{H})$ we have
  \begin{equation*}
    \mathbf{1} = \textstyle\bigvee_{u \in \mathbb{G}_0} \varphi([1_u]) \qquad \text{and}
    \qquad \varphi([1_u]) \cdot \varphi([1_v]) = \bot \text{ for $u \neq v$}\rlap{ ,}
  \end{equation*}
  so that the sets $s\bigl(\varphi([1_u])\bigr)$ are a partition of
  $\mathbb{H}_0$. We thus have a well-defined function
  $f \colon \mathbb{H}_0 \rightarrow \mathbb{G}_0$ determined
  by~\eqref{eq:19}; whereupon we obtain the assignations on morphisms
  required for a comorphism
  $f \colon \mathbb{G} \rightsquigarrow \mathbb{H}$ by the
  formula~\eqref{eq:20}. The comorphism axioms now follow easily from
  the homomorphism axioms for $\varphi$ together with the observation
  that $[a] \cdot [b] = [a \circ b]$ in $\mathrm{PBis}(\mathbb{G})$
  whenever $a$ and $b$ are composable maps. Finally, to see that
  $\mathrm{PBis}(f) = \varphi$, we observe that $\mathrm{PBis}(f)([a])
  = \varphi([a])$ by construction; now since for any $\alpha \in \mathrm{PBis}(\mathbb{G})$,
  we have $\alpha = \bigvee_{u \in s(\alpha)} [\alpha_u]$, and since
  both $\mathrm{PBis}(f)$ and $\varphi$ preserve joins, it follows that $\mathrm{PBis}(f)(\alpha) =
  \varphi(\alpha)$ for all $\alpha \in \mathrm{PBis}(\mathbb{G})$, as desired.
\end{proof}

It is not too hard to characterise the essential image of the
embedding functor $\cat{Grpd}_\mathrm{co} \to \cat{Ps\G rp}$; it
comprises the \emph{complete atomic} pseudogroups---those whose
partially ordered set of idempotents forms a complete atomic Boolean
algebra (i.e., a power-set lattice). Thus, our main result, concerning
the ``non-algebraic'' category of groupoids and comorphisms, can be
recast as one about the ``algebraic'' category of complete atomic
pseudogroups and pseudogroup homomorphisms; and
following~\cite{Lawson2013Pseudogroups}, we may recast the
generalisation of our main result to \'etale topological groupoids in
terms of more general pseudogroups.

There are a couple of points worth noting here. Firstly, when
translated into the language of pseudogroups, our main result states
that every extended inner automorphism of a complete atomic
pseudogroup $M$ is induced by conjugation (in the usual sense) by an
invertible element of the monoid $M$---indeed, such invertible
elements correspond to \emph{total} bisections of the corresponding
groupoid. So in this sense, our result fits into the pattern
established in~\cite{Bergman2012An-inner, Hofstra2018Isotropy}.

On the other hand, if we translate the proof of our main
Theorem~\ref{thm:1} into the language of complete atomic pseudogroups,
then it is still not a proof in the same mould
as~\cite{Bergman2012An-inner, Hofstra2018Isotropy}. If it were, then
the first step in determining the components of an extended inner
automorphism $\beta$ of $M$ would be to adjoin freely a new element
$x$ to $M$ and consider the component
$\beta_{\iota} \colon M[x] \rightarrow M[x]$ at the resulting
inclusion map $\iota \colon M \rightarrow M[x]$. This is quite
different from what is done in Proposition~\ref{prop:3}: in the
language of pseudogroups, the first step there is to consider an
atomic idempotent $u \in M$, and now consider the component
$\beta_\jmath$ corresponding to the homomorphism
\begin{align*}
  \jmath \colon M &\rightarrow \mathrm{PBij}(M_u)\\
  x & \mapsto x \cdot (\thg)
\end{align*}
into the pseudogroup of all partial bijections of
$M_u = \{m \in M : m^\ast m =
u\}$. It may be interesting to compare these approaches more
thoroughly, but would take us too far afield here.

\subsection{Extended partial inner automorphisms}
\label{sec:extend-part-inner}

In a pseudogroup $M$, \emph{any} element
$a \in M$ induces a conjugation map $c_a(x) = axa^\ast$. However,
$c_a \colon M \rightarrow M$ is not typically an automorphism of $M$,
nor even a well-defined homomorphism, since it does not preserve the
monoid unit $1$ unless $a \in M$ is genuinely invertible.

Nonetheless, we would like to think of $c_a$ as a \emph{partial}
automorphism of $M$, hoping for a result to the effect that every
extended inner partial automorphism of a pseudogroup is induced by
conjugation by an element. In the world of groupoids this would
translate into the statement that every extended partial automorphism
in $\cat{Grpd}_\mathrm{co}$ comes from conjugating by a partial
bisection.

Making this precise is delicate because, just as conjugation on a
pseudogroup does not give a pseudogroup homomorphism, so conjugation
on a groupoid by a partial bisection does not give a comorphism. Thus,
much as in Section~6 of~\cite{Bergman2012An-inner}, we must proceed in
an essentially \emph{ad hoc} manner.

\begin{Defn}
  \label{def:8}
  A \emph{partial automorphism} $\varphi \colon \mathbb{G} \tor
  \mathbb{G}$ of a groupoid $\mathbb{G}$ is given by full
  subcategories $s(\varphi), t(\varphi) \subseteq \mathbb{G}$ together
  with an isomorphism of groupoids $\varphi \colon s(\varphi)
  \rightarrow t(\varphi)$. Given a comorphism $f \colon \mathbb{G}
  \rightsquigarrow \mathbb{H}$ and partial automorphisms $\varphi
  \colon \mathbb{G} \tor \mathbb{G}$ and $\psi \colon
  \mathbb{H} \tor \mathbb{H}$, we say that
  \begin{equation}\label{eq:21}
    \cd{
      {\mathbb{G}} \ar[r]^-{\varphi}|-{+} \ar@{~>}[d]_{f} &
      {\mathbb{G}} \ar@{~>}[d]^{f} \\
      {\mathbb{H}} \ar[r]_-{\psi}|-{+}  &
      {\mathbb{H}}
    }
  \end{equation}
  is a \emph{commuting square} if:
  \begin{enumerate}[(i),itemsep=0.25\baselineskip]
  \item On objects, we have $u \in s(\psi)$ if and only if
    $f(u) \in s(\varphi)$; and for those $u$ where this does hold, we
    have that $\varphi(f(u)) = f(\psi(u))$ in $t(\varphi)$.
  \item For all $a \colon fu \rightarrow v$ in $s(\varphi)$, we have
    $\psi(f(a)_u) = f(\varphi(a))_{\psi(u)}$ in $t(\psi)$.
  \end{enumerate}
  
  Now by an \emph{extended inner partial automorphism} of $\mathbb{G}$, we
  mean a family of partial automorphisms $(\beta_f \colon \mathbb{H}
  \tor \mathbb{H})$, one for each comorphism $f \colon
  \mathbb{G} \rightsquigarrow \mathbb{H}$, such that for all comorphisms $f
  \colon \mathbb{G} \rightsquigarrow \mathbb{H}$ and $g \colon \mathbb{H}
  \rightsquigarrow \mathbb{K}$ we have a commuting square:
  \begin{equation*}
    \cd{
      {\mathbb{H}} \ar[r]^-{\beta_f}|-{+} \ar@{~>}[d]_{g} &
      {\mathbb{H}} \ar@{~>}[d]^{g} \\
      {\mathbb{K}} \ar[r]_-{\beta_{gf}}|-{+}  &
      {\mathbb{K}}\rlap{ .}
    }
  \end{equation*}
\end{Defn}
Using the fact that commuting squares of the form~\eqref{eq:21} stack
vertically and horizontally to give commuting squares, we can now
follow through the same argument as before, \emph{mutatis mutandis},
to show that:
\begin{Thm}
  The monoid of extended partial inner automorphisms of a groupoid
  $\mathbb{G}$ is isomorphic to the monoid of partial bisections
  $\mathrm{PBis}(\mathbb{G})$. The extended inner automorphism
  corresponding to $\alpha \in \mathrm{PBis}(\mathbb{G})$ is given by
  the family of partial automorphisms
  $\bigl(c_{f\alpha} \colon \mathbb{H} \rightsquigarrow
  \mathbb{H}\bigr)$ as $f$ ranges over comorphisms
  $\mathbb{G} \rightsquigarrow \mathbb{H}$.
\end{Thm}
The details are sufficiently similar that we leave them to the
interested reader to reconstruct.

\bibliographystyle{acm} \bibliography{bibdata}


%
%
%
%
%

\end{document}